\documentclass[12pt,reqno]{article}
\usepackage{amssymb,amscd,amsmath,amsthm,color}
\newtheorem{theorem}{Theorem}[section]

\newtheorem{cor}[theorem]{Corollary}
\newtheorem{prop}[theorem]{Proposition}

\usepackage{enumerate,amssymb}
\theoremstyle{definition}

\theoremstyle{remark}
\newtheorem{remark}[theorem]{Remark}

\numberwithin{equation}{section}

\def\bM{\mathbb{M}}

\begin{document}
\baselineskip=15pt

\title{Diagonal and off-diagonal  blocks of positive definite partitioned matrices}

\author{ Jean-Christophe Bourin{\footnote{Funded by the ANR Projet (No.\ ANR-19-CE40-0002) and by the French Investissements
 d'Avenir program, project ISITE-BFC (contract ANR-15-IDEX-03).}} \,  and Eun-Young Lee{\footnote{This research was supported by
Basic Science Research Program through the National Research
Foundation of Korea (NRF) funded by the Ministry of
Education (NRF-2018R1D1A3B07043682)}  }    }

\date{ }

\maketitle

\vskip 10pt\noindent
{\small 
{\bf Abstract.} We obtain  new relations for  the  blocks of a positive semidefinite matrix $\begin{bmatrix} A& X \\ X^* &B\end{bmatrix}$ partitioned into four blocks in $\bM_n$. A consequence is
\begin{equation}\label{abs1}
|X+X^*| \le  A+B +\frac{1}{4}V(A+B)V^*
\end{equation}
for  some unitary $V\in \bM_n$. Here, for $n\ge  2$, the constant $1/4$ cannot be replaced by  any smaller one.   Several eigenvalue inequalities for  general matrices  follow from our result. We can also derive from \eqref{abs1} some triangle type inequalities, for instance
 for three contractions,
\begin{equation}\label{abs2}
|S+T+R| \le \frac{3}{4}I +|S|+|T|+|R|
\end{equation}
where $I$ stands for the identity. We conjecture that the constant $3/4$ is optimal.

}

\section{Arithmetic-geometric mean inequalities}

 Inequalities for positive (semi-definite) matrices  partitioned into four blocks provide a lot
of important theorems in Matrix analysis. Let us review in this introduction a few well-known examples related to matrix versions of the AGM inequality; the reader may find much more examples  in the  books \cite{Bhatia}, \cite{HP}. 

Let $\begin{bmatrix} A &X \\
X^* &B\end{bmatrix} $ be a  positive matrix  partitioned into four blocks in $\bM_n$, the space of $n$-by-$n$ matrices. An elementary inequality pointed out by Tao \cite{Tao} states that, for $j=1,\ldots, n$ 
\begin{equation}\label{eqT}
2\lambda_j^{\downarrow}(|X|) \le \lambda_j^{\downarrow}\left(\begin{bmatrix} A &X \\
X^* &B\end{bmatrix} \right)
\end{equation}
where $|X|$ is the positive part in the polar decomposition $X=U|X|$ and $\lambda_j^{\downarrow}(\cdot)$ stand for the eigenvalues arranged in the non-increasing order. In despite of its simplicity, Tao's result entails the famous inequality of Bhatia and Kittaneh \cite{BK1}: Let $A,B\in\bM_n$. Then there exists a unitary matrix $U\in\bM_n$ such that
\begin{equation}\label{eqBK}
|AB| \le U\frac{A^*A + BB^*}{2}U^*.
\end{equation}
Three remarkable extensions of this classical inequality have been given by Ando \cite{A2}, Audenaert \cite{Aud} (see also \cite[Sec. 5]{BK2}), and Drury \cite{D}.

Another kind of arithmetic geometric mean inequality follows from the positive partitioned matrix
$$
\sum_{i=1}^n\begin{bmatrix}  A^*_i \\
 B_i^* \end{bmatrix} 
\begin{bmatrix}  A_i &  B_i \end{bmatrix} 
=\begin{bmatrix} \sum_{i=1}^n A^*_iA_i & \sum_{i=1}^n A^*_iB_i \\
\sum_{i=1}^n B^*_iA_i & \sum_{i=1}^n B_i^*B_i\end{bmatrix} 
$$
where $A_i,B_i\in\bM_d$, $i=1,\ldots, n$.
Using the unitary congruence implemented by
$
\begin{bmatrix} I &0 \\
0 &V\end{bmatrix}
$
where $V$ is the unitary part in the polar decomposition $\sum_{i=1}^n B^*_iA_i =V|\sum_{i=1}^n B^*_iA_i |$, we obtain that 
$$
\begin{bmatrix} \sum_{i=1}^n A^*_iA_i & \left|\sum_{i=1}^n B^*_iA_i\right| \\
\left|\sum_{i=1}^n B^*_iA_i\right| & V^*\left(\sum_{i=1}^n B_i^*B_i\right)V\end{bmatrix} 
$$
is positive too, so that the maximal property  of the geometric mean $\#$ entails
\begin{equation}\label{eq1}\left|\sum_{i=1}^n B^*_iA_i\right|  \le \left(\sum_{i=1}^n A^*_iA_i \right)\# V^*\left(\sum_{i=1}^n B_i^*B_i\right)V,
\end{equation}
hence,
$$\left|\sum_{i=1}^n B^*_iA_i\right|  \le \frac{\sum_{i=1}^n A^*_iA_i  + V^*\left(\sum_{i=1}^n B_i^*B_i\right) V}{2}.
$$
We refer to   \cite{A1} and  \cite{Bhatia2} for a background on $\#$. Since  fundamental 1979 works of Ando  \cite{A1}, the operator geometric means is a wonderful tool when dealing with matrix inequalities.

In particular, letting $n=1$ in the last inequality,  we may complement \eqref{eqBK} with
$$
|B^*A| \le \frac{A^*A + V^*B^*BV}{2}
$$
and this was used   \cite[Lemma 2.5 ]{BH} in the context of von Neumann algebras. Since the geometric mean of two positive (semi-definite)  matrices $X,Y$ satisfies
\begin{equation}\label{eq-geom}
X\#Y =X^{1/2}WY^{1/2}
\end{equation}
for some unitary $W$  (if $X,Y$ are positive definite, then $W$ is unique),  a standard log-majorisation (Horn's inequality) on \eqref{eq1} gives
the Bhatia-Davis \cite{BD} type 
Schwarz inequality
\begin{equation}\label{eq-BD}
\left\| \, \left|\sum_{i=1}^n B^*_iA_i\right|^{\alpha} \, \right\|^2 \le 
\left\| \,\left|\sum_{i=1}^n B^*_iB_i\right|^{\alpha}\, \right\| \left\| \, \left|\sum_{i=1}^n A^*_iA_i\right|^{\alpha} \, \right\|
\end{equation}
for all $\alpha>0$ and all unitarily invariant norms. To illustrate the scope of this inequality, consider the normalized trace norm on $\bM_n$, $\| X\|_{\tau}:=n^{-1}{\mathrm{Tr\,}} |X|$. Then, using
$$
{\det}^{1/n} |X|=\lim_{\alpha \searrow 0}  \| \,|X|^{\alpha}\|_{\tau}^{1/\alpha},
$$
we infer from \eqref{eq-BD} the determinantal Schwarz inequality
$$
{\det}^2\left|\sum_{i=1}^n B^*_iA_i\right| \le \det \left(\sum_{i=1}^n B^*_iB_i\right) \det  \left(\sum_{i=1}^n A^*_iA_i\right).
$$

This trick in \eqref{eq1} with the unitary part of an off-diagonal block $X=V|X^*|$,
\begin{equation}\label{trick}
\begin{bmatrix} A&X^* \\ X&B
\end{bmatrix}\ge 0 \Longrightarrow |X|\le A\#(V^*BV)\le \frac{A+V^*BV}{2},
\end{equation}
goes back to the paper \cite[Eq. (2.8)]{BR}.
We used it in \cite{Lee} to study PPT matrices. In two more recent papers \cite{BL2018} \cite{BL2019}, it is employed to get striking inequalities for positive linear maps on normal operators, and related exotic eigenvalue estimates. For instance, for the Shur product $S \circ A$ of a positive matrix $S$ with   a contraction $A$, we 
have \cite[Corollary 3.2]{BL2019}
$$
\lambda_3^{\downarrow}(S \circ A) \le  \delta_2^{\downarrow}(S)
$$
where $\delta^{\downarrow}_2(S)$ denotes the second largest diagonal entries of $S$.

Thus, results for blocks of positive partitioned matrices have many applications, sometimes surprizing and for which no different proof are easily available. It may be of some interest to have more results for these blocks. We shall  give a few new ones in this paper, still via the trick \eqref{trick}. We will get again some nice inequalities, for instance for the Schur product of two positive matrices,
\begin{equation}\label{intro}
\lambda_3^{\downarrow} (|AB\circ BA|) \le \lambda_2^{\downarrow}(A^2\circ B^2).
\end{equation}

Section 2 gives our main result, Theorem \ref{mainThm}, which establishes several  sharp operator inequalities including  \eqref{abs1}. This theorem is in the  continuation of our previous works \cite{BL2018}, \cite{BL2019} where comparison between diagonal blocks and off-diagonal blocks leaded to remarkable estimates.  

Several eigenvalue estimates such as \eqref{intro} are derived in Section 3.  We will  also get generalizations of some eigenvalue estimates by Audeh, Bhatia, and Kittaneh.

In Section 4, we note that Theorem \ref{mainThm} completes a result for Schur product of normal operators first proved in \cite{BL2018}. We  also   point out  a special case of \eqref{abs1} showing that, given two Hermitian matrices such that $\pm S\le T$,  the sharp inequality
\begin{equation}\label{int}
|S| \le  T +\frac{1}{4}VTV^*
\end{equation}
holds for some unitary matrix $V$. 
A simple consequence of \eqref{int} yields some  triangle type inequalities for contractions such as \eqref{abs2}.

Though the paper is written in the setting of $\bM_n$, our results
have obvious extensions to operators in a von Neumann algebra ${\mathcal{M}}$,  relations such as  \eqref{abs1}  holding in $\bM_2({\mathcal{M}})$.

\section{Sharp inequalities for blocks of positive matrices}

\vskip 10pt
We use the symbol $\diamond$  to denote either the sum, or minus  or Schur product sign. By definition, {\it a symmetry matrix} is both Hermitian and unitary. So, a symmetry $V\in\bM_n$ is characterized by the property $V=V^*=V^{-1}$.

\vskip 10pt
\begin{theorem}\label{mainThm}  Let $\begin{bmatrix} A &X \\
X^* &B\end{bmatrix} $ be a positive matrix  partitioned into four blocks in $\bM_n$, Fix $\diamond\in\{+,\circ\} $. Then, for some   unitary $V_{\diamond}\in\bM_n$,
$$
|X\diamond X^*| \le A\diamond B +\frac{1}{4}V_{\diamond}(A\diamond B)V^*_{\diamond},
$$
where the constant $1/4$ cannot be replaced by any smaller one. 

Furthermore, we can require that $V_{\diamond}$ is a symmetry  and that we also have
$$
|X\diamond X^*| \le (A\diamond B)\#V_{\diamond}(A\diamond B)V_{\diamond}.
$$

\end{theorem}

\vskip 10pt
Note that in general $X\circ X^*$ is not positive, even for $X=X^*$.

\vskip 10pt
\begin{proof} First we consider $\diamond=+$. Since both
\begin{equation}\label{pair}
\begin{bmatrix} A &X \\
X^* &B
\end{bmatrix} \quad{\mathrm{and}} \quad
\begin{bmatrix} B &X^* \\
X &A
\end{bmatrix} 
\end{equation}
are positive, the second matrix being derived from the first one with the unitary congruence implemented by $\begin{bmatrix} 0 &I  \\
I & 0
\end{bmatrix}$, so is their sum
$$
\begin{bmatrix} A+ B &X + X^* \\
X+X^* & A+ B
\end{bmatrix}.
$$
Arguing as in the proof of \eqref{eq1}, we obtain a positive block-matrix
$$
\begin{bmatrix} A+ B &|X + X^*| \\
|X+X^*| & V_{+}(A+ B)V_{+}^*
\end{bmatrix}.
$$
 with $V_+=V_+^*$, the unitary part in the polar decomposition of the Hermitian matrix $X+X^*=V_{+}|X+X^*|$. By the maximal property of $\#$
we get
$$
|X +X^*| \le (A+ B)\#V_{+}(A+ B)V_{+}.
$$
This proves the second inequality of the theorem for $\diamond =+$. 
From the operator AGM inequality
\begin{equation}\label{agm}
X\#Y =(2s^{-1}X)\#(2^{-1}sY) \le s^{-1}X+\frac{s}{4}Y
\end{equation}
for $s>0$ we then infer the first inequality of the theorem (with $s=1/4$), for $\diamond =+$. This first inequality is equivalent to the eigenvalue inequality
\begin{equation}\label{eig1}
\lambda_j^{\downarrow}\left(|X+X^*|-(A+B)\right) \le \frac{1}{4}\lambda_j^{\downarrow}\left(A+B\right).
\end{equation}
Now, consider for $t>0$,
\begin{equation}\label{niceex}
\begin{bmatrix} A  &X  \\
X^* &  B
\end{bmatrix}=
\begin{bmatrix} \begin{pmatrix}  t&0 \\ 0&0\end{pmatrix}& \begin{pmatrix}  0&1 \\ 0&0\end{pmatrix} \\
 \begin{pmatrix}  0&0 \\ 1&0\end{pmatrix} & \begin{pmatrix}  0&0 \\ 0&t^{-1}\end{pmatrix}
\end{bmatrix}
\end{equation}
so that \eqref{eig1} for $j=1$ and $t=1/2$ reads as the equality
$$
\lambda_1^{\downarrow}\left(\begin{bmatrix} 1-t&0\\ 0&1-1/t\end{bmatrix}\right) =1/2=
 \frac{1}{4}\lambda_1^{\downarrow}\left(\begin{bmatrix} t&0\\ 0&1/t\end{bmatrix}\right)
$$
showing that the constant $1/4$ cannot be diminished in \eqref{eig1}, and thus is optimal in the first inequality of our theorem.

We turn to the proof for $\diamond =\circ$,  similar to the $+$ case, as
$$
\begin{bmatrix} A\circ B &X\circ X^* \\
X\circ X^* & A\circ B
\end{bmatrix}
$$
is positive since it is the Shur product of the two positive matrices in \eqref{pair}. Thus, with
$V_{\circ}=V_{\circ}^*$  the unitary part in the polar decomposition of the Hermitian matrix $X\circ X^*=V_{\circ}|X\circ X^*|$ we see that
$$
\begin{bmatrix} A\circ B &|X \circ X^*| \\
|X\circ X^*| & V_{\circ}(A\circ B)V_{\circ}
\end{bmatrix}
$$ is positive. Using again the maximal property of $\#$ completes the proof of the second inequality of the theorem for $\diamond =\circ$. The first inequality follows  from \eqref{agm} with $s=1/4$. This first inequality is equivalent to the eigenvalue inequality
\begin{equation}\label{eig2}
\lambda_j^{\downarrow}\left(|X\circ X^*|-A\circ B\right) \le \frac{1}{4}\lambda_j^{\downarrow}\left(A\circ B\right).
\end{equation}
Next, choose for $t>0$,
$$
\begin{bmatrix} A  &X  \\
X^* &  B
\end{bmatrix}=
\begin{bmatrix} \begin{pmatrix}  t^{1/2}&0 \\ 0&t^{-1/2}\end{pmatrix}& \begin{pmatrix}  0&1 \\ 1&0\end{pmatrix} \\
 \begin{pmatrix}  0&1 \\ 1&0\end{pmatrix} & \begin{pmatrix}  t^{1/2}&0 \\ 0&t^{-1/2}\end{pmatrix}
\end{bmatrix}
$$
so that \eqref{eig2} for $j=1$ and $t=1/2$ reads as the equality
$$
\lambda_1^{\downarrow}\left(\begin{bmatrix} 1-t&0\\ 0&1-1/t\end{bmatrix}\right) =1/2=
 \frac{1}{4}\lambda_1^{\downarrow}\left(\begin{bmatrix} t&0\\ 0&1/t\end{bmatrix}\right)
$$
showing that the constant $1/4$ cannot be diminished in \eqref{eig2}, and thus is optimal in the first inequality of our theorem for the Schur product too.
\end{proof}

We have a variation of Theorem \ref{mainThm} for the minus sign as follows.

\vskip 10pt
\begin{cor}\label{corminus}  Let $\begin{bmatrix} A &X \\
X^* &B\end{bmatrix} $ be a positive matrix  partitioned into four blocks in $\bM_n$. Then, for some   symmetry $V_{-}\in\bM_n$,
$$
|X- X^*| \le (A+ B)\#V_{-}(A+ B)V_{-}
$$
and
$$
|X- X^*| \le (A+ B) +\frac{1}{4}V_{-}(A+ B)V_{-}
$$
where the constant $1/4$ is optimal.
\end{cor}

\vskip 10pt 
\begin{proof}
We apply the $+$ sign case of the theorem to the partitioned matrices
$$
\begin{bmatrix} A &iX \\ -iX^*&B\end{bmatrix} \quad{\mathrm{and}} \quad
\begin{bmatrix} B &-iX^* \\
iX &A
\end{bmatrix} 
$$
which are positive  since
 derived from the first matrix in \eqref{pair} with the unitary congruence implemented by
$\begin{bmatrix} I&0 \\ 0&iI\end{bmatrix}$. Thus $|iX-iX^*|=|X-X^*|$ and $V_{-}$ is the 
unitary part in the polar decomposition $iX-iX^*=V_{-}|X-X^*|$.
Note that the  proof  uses the complex number $i$, however, if $X$ has only real entries, then so has $V_{-}$. To show that $1/4$ is optimal in the second inequality of the corollary, we use again the partitioned matrix \eqref{niceex} with $t=1/2$.
\end{proof}

Of course, from the matrix AGM inequality, we  also have the following   consequence of Theorem \ref{mainThm}.

\vskip 10pt
\begin{cor}\label{cormean}  Let $\begin{bmatrix} A &X \\
X^* &B\end{bmatrix} $ be a positive matrix  partitioned into four blocks in $\bM_n$ and fix $\diamond\in\{+,-\}$. Then, for some symmetry $V_{\diamond}\in\bM_n$,
$$
|X\diamond X^*| \le \frac{(A+B)+V_{\diamond}(A+ B)V_{\diamond}}{2}
$$
\end{cor}

\section{Results from the geometric mean}

This section is devoted to easy consequences of the second inequality of Theorem \ref{mainThm} involving the geometric mean. By using this  second inequality combined with \eqref{eq-geom} and a basic inequality of Weyl, we   get the following corollary.

\vskip 10pt
\begin{cor}\label{corShur2}  Let $\begin{bmatrix} A &X \\
X^* &B\end{bmatrix} $ be a positive matrix  partitioned into four blocks in $\bM_n$, and let $\diamond\in\{+,\circ \, -\} $. Then, for all nonnegative integers $j$ and $k$,
$$
\lambda^{\downarrow 2}_{1+j+k}(|X\diamond X^*|) \le \lambda_{1+j}^{\downarrow}(A \diamond B)\lambda_{1+k}^{\downarrow}(A \diamond B)
$$
\end{cor}

\vskip 10pt
In particular,  for $j=0,1,\ldots, n-1$, we have $
\lambda_{2j+1}^{\downarrow}(|X+X^*|) \le \lambda_{j+1}^{\downarrow}(A+ B).
$
This inequality applied   to the positive block-matrix
\begin{equation}\label{spe}
\begin{bmatrix} A^*A & A^*B\\ B^*A
& B^*B\end{bmatrix} 
\end{equation}
yields an inequality of Bhatia-Kittaneh \cite[Proposition 6.2]{BK2}. The next corollary, a special case of Corollary \ref{corShur2} with \eqref{spe}, says  more.

\vskip 10pt
\begin{cor}\label{corShur4}  Let $A$ and $B$ be two  matrices in $\bM_n$, and let $\diamond\in\{+,\circ \} $. Then, for all nonnegative integers $j$ and $k$,
$$
\lambda^{\downarrow 2}_{1+j+k}(|A^*B\diamond B^*A|) \le \lambda_{1+j}^{\downarrow}(A^*A \diamond B^*B)\lambda_{1+k}^{\downarrow}(A^*A \diamond B^*B)
$$
\end{cor}

\vskip 10pt
The case $j=k=1$ was already stated in \eqref{intro}  for $\diamond= \circ$ and two positive matrices $A$ and $B$.

Some classical log-majorisation principle  allow to derive from Theorem \ref{mainThm} the next norm estimates.

\vskip 10pt
\begin{cor}\label{corShur3}  Let $\begin{bmatrix} A &X \\
X^* &B\end{bmatrix} $ be a positive matrix  partitioned into four blocks in $\bM_n$, and let $\diamond\in\{+,\circ \} $. Then, for all $\alpha>0$ and all unitarily invariant norms,
$$
\left\| |X\diamond X^*|^{\alpha}\right\| \le  \left\|(A\diamond B)^{\alpha}\right\|.
$$
\end{cor}

\vskip 10pt
\begin{cor}\label{corShur5}  Let $A$ and $B$  be two  matrices in $\bM_n$, and let $\diamond\in\{+,\circ \} $. Then, 
 for all $\alpha>0$ and all unitarily invariant norms,
$$
\left\| |A^*B\diamond B^*A|^{\alpha}\right\| \le \left\| |A^*A\diamond B^*B|^{\alpha}\right\|.
$$
\end{cor}

\vskip 10pt
Denote by $\delta^{\downarrow}_{k}(H)$, $k=1,2,\ldots,n$,  the $k$-largest diagonal entries of a Hermitian matrix $H\in\bM_n$.

\vskip 10pt
\begin{cor}\label{corShur6}  For all matrices $Z\in\bM_n$ and
   $j=0,1,\ldots, n-1$,
$$
 \lambda_{1+2j}^{\downarrow}\left(|Z\circ Z^*|\right) \le \min\{\delta^{\downarrow}_{1+j}(Z^*Z), \delta^{\downarrow}_{1+j}(ZZ^*)\}.
$$
\end{cor}

\vskip 5pt
In particular, we have
$
 \lambda_{3}^{\downarrow}\left(|Z\circ Z^*|\right) \le \delta^{\downarrow}_{2}(Z^*Z).
$

\vskip 10pt
\begin{proof} Consider the positive matrices
$$\begin{bmatrix} I &Z \\
Z^* &Z^*Z\end{bmatrix} \quad{\mathrm{and}}\quad
\begin{bmatrix} I &Z^* \\
Z &ZZ^*\end{bmatrix}
$$
and  apply with $j=k$ Corollary \ref{corShur2} for $\diamond=\circ$.
\end{proof}

\vskip 10pt
 Let $Z\in\bM_n$. The polar decomposition $Z=|Z^*|^{1/2}V|Z|^{1/2}$ shows that the block matrix
\begin{equation}\label{trick2}
\begin{bmatrix} |Z^*| &  Z \\
Z^* &|Z|\end{bmatrix} 
\end{equation}
is positive. The next two corollaries follows from our previous results applied to this positive block matrix. The weak log-majorisation relation in $\bM_n^+$, $X\prec_{w\log} Y$, is equivalent to  
$
\| X^{\alpha}\| \le \| Y^{\alpha}\|
$
for all $\alpha>0$ and all unitarily invariant norms.

\vskip 10pt
\begin{cor}\label{corA} Let $Z\in\bM_n$  and let $\diamond\in\{+,\circ \} $. Then,  $|Z\diamond Z^*|\prec_{wlog} |Z| \diamond |Z^*|$.
\end{cor}

\vskip 10pt
\begin{cor}\label{corB} Let $Z\in\bM_n$  and let $\diamond\in\{+,\circ \} $. Then for some symmetry $V_{\diamond}\in\bM_n$,
$$|Z\diamond Z^*| \le (|Z|\diamond |Z^*|) \#V_{\diamond}(|Z|\diamond |Z^*|)V_{\diamond}$$
\end{cor}

\vskip 10pt
Corollary \ref{corA} was first noted in \cite[Corollary 2.13]{BL2016}. The stronger Corollary \ref{corB}  was first  obtained in \cite[Corollary 3.6]{BL2018}.

We close this section by some direct consequences of \eqref{trick}.

\vskip 10pt
\begin{prop}\label{prop0}  Let $\begin{bmatrix} A &X \\
X^* &B\end{bmatrix} $ be a positive matrix  partitioned into four blocks in $\bM_n$. Then, for some unitary  $U\in\bM_{2n}$,
$$
|X| \oplus |X| \le (A\oplus B)\#U(A\oplus B)U^*.
$$
\end{prop}
 
\vskip 5pt
\begin{proof} From \eqref{trick}, we have $|X^*|\le A\#V_1BV_1^*$ and $|X|\le B\#V_2AV_2^*$ for some unitaries $V_1,V_2$. This gives the proposition since $|X|$ and $|X^*|$ are unitarily congruent.
\end{proof}

\vskip 5pt
 Proposition \ref{prop0} combined with Weyl's inequalities entails  the following result of Audeh and Kittaneh \cite{AK} :
$$
\lambda_{j+1}^{\downarrow}(|X|)=\lambda_{2j+1}^{\downarrow}(|X|\oplus|X|)\le \frac{1}{2} \left\{\lambda_{j+1}^{\downarrow}(A\oplus B) +\lambda_{j+1}^{\downarrow}(A\oplus B)\right\}, \quad j=0,\ldots,n-1
$$
so we get the main result of \cite{AK},
$$
\lambda_{j+1}^{\downarrow}(|X|) \le \lambda_{j+1}^{\downarrow}(A\oplus B)
, \quad j=0,\ldots,n-1.
$$

In fact Proposition \ref{prop0} and Weyl's inequalities show that we may complete Audeh-Kittaneh's result with a large number of inequalities, stated in our next corollary.
 
\vskip 5pt
\begin{cor}\label{cor-akext}Let $\begin{bmatrix} A &X \\
X^* &B\end{bmatrix} $ be a positive matrix  partitioned into four blocks in $\bM_n$.
Then, for all integers nonnegative integers $j,k,l$ such that   $2j=k+l$,
$$
\lambda_{j+1}^{\downarrow}(|X|)\le  \left\{\lambda_{k+1}^{\downarrow}(A\oplus B) \lambda_{l+1}^{\downarrow}(A\oplus B)\right\}^{1/2}.
$$
\end{cor}

Here we adhere to the convention that, for a positive matrix $A\in\bM_n$, $\lambda_i(A):=0$ whenever $i>n$. This natural convention avoid to add an extra assumption $j<n$ in Corollary \ref{cor-akext}. We keep this convention in the next statement, an extension of \cite[Theorem 6.3]{BK2}  (the case $k=l=j$ in Corollary \ref{cor-akext2}).

\vskip 5pt
\begin{cor}\label{cor-akext2} Let $A,B\in\bM_n$. Then, for all  non-negative integers $j,k,l$ such that $2j=k+l$,
$$
\lambda_{j+1}^{\downarrow}(|A+B|)\le  \left\{\lambda_{k+1}^{\downarrow}\left[(|A|+|B|)\oplus (|A^*|+|B^*|)\right] \cdot \lambda_{l+1}^{\downarrow}\left[(|A|+|B|)\oplus (|A^*| + |B^*|)\right]\right\}^{1/2}.
$$
\end{cor}

\vskip 10pt\noindent
\begin{proof} Applying Corollary  \ref{cor-akext} to the matrix (positive by \eqref{trick2})
\begin{equation*}
\begin{bmatrix} |A^*| +|B^*|&  A+B \\
A^*+B^* &|A| + |B|\end{bmatrix} 
\end{equation*}
yields the result.
\end{proof}

\section{Optimality of 1/4 and triangle inequalities}

We may use our main theorem to derive a result for the Schur product of normal matrices.  In \cite[Corollary 3.5]{BL2018} we obtained the following result by using positive linear maps.

\vskip 5pt
\begin{prop}\label{corschurnormal} If $A,B\in\bM_n$ are normal,  then, for some unitary $V\in\bM_n$,
\begin{equation*}
|A\circ B| \le |A|\circ|B|+ \frac{1}{4}V(|A|\circ|B|)V^*
\end{equation*}
where the constant $1/4$ is optimal.
\end{prop}

We complete this estimate with a new one.

\vskip 5pt
\begin{prop}\label{corschurnormal2} If $A,B\in\bM_n$ are normal,  then, for some unitary $V\in\bM_n$,
\begin{equation*}
\frac{\left|A\circ B+A^*\circ B^*\right|}{2} \le  |A|\circ|B|+ \frac{1}{4}V(|A|\circ|B|)V^*
\end{equation*}
where the constant $1/4$ is optimal.
\end{prop}

This is a direct consequence of Theorem \ref{mainThm} applied to the matrix
$$
\begin{bmatrix} |A|\circ |B| & A\circ B \\ A^*\circ B^*& |A|\circ |B| \end{bmatrix}
$$
Indeed, this matrix is positive as the Schur product of two matrices of the form \eqref{trick2}. To see that the constant $1/4$ is optimal in theses two propositions,  pick
$$
A=\begin{pmatrix} 2& 1 \\ 1& 1/2 \end{pmatrix}, \quad B=\begin{pmatrix} 0& 1 \\ 1& 0 \end{pmatrix}.
$$

We can use Theorem \ref{mainThm} to get further sharp  inequalities involving the constant 1/4. The next corollary deals with 
 two Hermitian matrices $S$ and $T$ such that  $S\le T$ and $-S\le T$, in short $\pm S\le T$. Contrarily to the obvious commutative case, we cannot infer, for $j>1$, the eigenvalue inequality, $\lambda_j^{\downarrow}(|S|)\le \lambda_j^{\downarrow}(T)$. For instance the ratio  $\lambda_2^{\downarrow}(|S|)/\lambda_2^{\downarrow}(T)$ can be arbitrarily large, 
as shown with $S:=P-Q$, $T:=P+Q$, where $P,Q$ are the rank one projections
$$
P=\begin{pmatrix} 1&0\\ 0&0 \end{pmatrix}, \quad Q=\begin{pmatrix} \cos^2a &\cos a\sin a\\ \cos a\sin a&\sin^2a \end{pmatrix}.
$$
Then, 
$$
\lim_{a\to 0^+}\frac{\lambda_2^{\downarrow}(|P-Q|)}{\lambda_2^{\downarrow}(P+Q)}=\lim_{a\to 0^+}\frac{\sin a}{1-\cos a} =+\infty.
$$
 However, we can state :
 
\begin{cor} If $S,T$ are Hermitian and satisfy $\pm S \le T$, then, for some unitary $V$,
$$
|S|\le T +\frac{1}{4}VTV^*.
$$
The constant $1/4$ is the smallest possible one, and we can require that $V$ is a symmetry.
\end{cor}

\begin{proof} The assumption entails that
$$
\begin{bmatrix} T+S&  0 \\
0 &T-S\end{bmatrix} 
$$
is positive. The unitary congruence implemented by
$$
\frac{1}{\sqrt{2}}\begin{bmatrix} I&  -I \\
I &I\end{bmatrix} 
$$
then shows that
$$
\begin{bmatrix} T&  S \\
S &T\end{bmatrix} 
$$
is positive too. Applying Theorem \ref{mainThm} to this matrix yields the inequality of our corollary for $\diamond=+$.
Choosing
$$
T=\begin{pmatrix} 2&  0 \\
0 &1/2\end{pmatrix}, \quad 
S=\begin{pmatrix} 0&  1 \\
1 &0\end{pmatrix}
$$
shows that the constant $1/4$ cannot be diminished.
\end{proof}

\vskip 5pt
\begin{cor}\label{cortriangle} Let
$A_1,\ldots, A_k$  be contractions in $\bM_n$, $k>1$. Then,  
$$
   \left| \sum_{j=1}^k A_j\right| \le  \frac{k}{4}I +\sum_{j=1}^k |A_j|.
$$
If $k$ is even and $n\ge 3$, then the constant $k/4$ is sharp.
\end{cor}

\vskip 5pt
\begin{proof} First, suppose that every matrix $A_i$ is Hermitian. Then $$\pm(A_1+\cdots+ A_k)\le \sum_{j=1}^k |A_j|$$ and the previous corollary  yields
$$
\left| \sum_{j=1}^k A_i\right| \le  \frac{1}{4}V(\sum_{j=1}^k |A_j|)V^* + \sum_{j=1}^k |A_j|.
$$
for some unitary $V$.
Since $|A_j|\le I$ for all $j$, we get the inequality of the corollary for $k$ Hermitian matrices,
$$
   \left| \sum_{j=1}^k A_j\right| \le  \frac{k}{4}I +\sum_{j=1}^k |A_j|.
$$
 To get the  general case, we apply the above inequality to the  $k$ Hermitian matrices,
\begin{equation}\label{abv}
\begin{bmatrix}
0&A_1 \\ A_1^*&0
\end{bmatrix}, \ldots, \begin{bmatrix}
0&A_k \\ A_k^*&0
\end{bmatrix}.
\end{equation}
A direct computation of the operator absolute values in \eqref{abv} completes the proof. 

 To check the sharpness of $k/4$ when $k$ is even and $n=3$, it suffices to consider the case $k=2$ and to find two contractions $C_1,C_2\in\bM_3$ such that the inequality
\begin{equation}\label{twocont}
|C_1+C_2|\le \frac{1}{2}I +|C_1|+|C_2|
\end{equation}
holds,  $1/2$ being optimal. For the general case $k=2l$, we then just pick $A_1,\ldots,A_l=C_1$, and $A_{l+1},\ldots A_{2l}=C_2$. To build up $C_1,C_2$, 
we follow an anonymous referee (see the next remark). Consider the two orthogonal ($PQ=QP=0$)  rank one projections 
\begin{equation}\label{exref}
P=\frac{1}{2}\begin{pmatrix} 1 & 1/2 & \sqrt{3}/{2} \\
1/2 & 1/4 & \sqrt{3}/4 \\
 \sqrt{3}/2  & \sqrt{3}/4 & 3/4
\end{pmatrix}, \quad 
Q=\frac{1}{2}\begin{pmatrix} 1 & -1/2 & -\sqrt{3}/2 \\
-1/2 & 1/4 & \sqrt{3}/4 \\
 -\sqrt{3}/2  & \sqrt{3}/4 & 3/4
\end{pmatrix}.
\end{equation}
Define 
 $C_1=P-Q$ and $C_2=V(P-Q)V^*$ where $$V=\begin{pmatrix} 1&0&0 \\ 0&1&0 \\ 0&0&-1\end{pmatrix}.$$
Then,
$$
|C_1+C_2|=\begin{pmatrix} 1&0&0 \\ 0&1&0 \\ 0&0&0\end{pmatrix}, \quad |C_1|+|C_2|=\begin{pmatrix} 2&0&0 \\ 0&1/2&0 \\ 0&0&3/2\end{pmatrix},
$$
showing that we cannot reduce $1/2$ in  \eqref{twocont}. \end{proof}

\begin{remark}\label{referee} The first version of this paper (still available on arXiv) proved optimality of the constant $k/4$ ($k=2l$) in Corollary \ref{cortriangle} only for $n\ge 18$ (for real or complex matrices), by using a proof with positive linear maps.  A generous referee pointed out  the above  example \eqref{exref} in $\bM_3$, allowing to considerably improve the condition on the dimension.   We are quite grateful   for this contribution. 
\end{remark}

We conjecture that the constant $k/4$ in Corollary \ref{cortriangle} is also optimal for all odd integers  $k>1$.

\vskip 5pt

Jean-Christophe Bourin

Université de Franche-Comté, CNRS, LmB, F-25000 Besançon, France

Email: jcbourin@univ-fcomte.fr

\vskip 15pt

Eun-Young Lee

Department of mathematics,   Kyungpook National University,   Daegu 702-701, Korea

Email: eylee89@knu.ac.kr

\end{document}